\documentclass[11pt,a4paper]{amsart}
\usepackage[utf8]{inputenc}
\usepackage[english]{babel}
\usepackage{amsmath}
\usepackage{amsfonts}
\usepackage{amssymb}
\usepackage{amsthm}
\usepackage{mathrsfs}
\usepackage{color}
\usepackage{hyperref}

\advance\hoffset by -1.5cm

\theoremstyle{definition} \newtheorem{definition}{Definition}[section]
\theoremstyle{definition} 
\theoremstyle{remark}     \newtheorem{remark}[definition]{Remark}
\theoremstyle{plain}      \newtheorem{theorem}[definition]{Theorem}
\theoremstyle{plain}      \newtheorem{proposition}[definition]{Proposition}
\theoremstyle{plain}      
\theoremstyle{plain}      \newtheorem{lemma}[definition]{Lemma}
\theoremstyle{plain}      
\theoremstyle{plain}      
\theoremstyle{definition} 

  \newcommand{\B}[1]{\mathcal{B}(#1)}
  \newcommand{\norme}[1]{\left\| #1 \right\|}
  \newcommand{\abs}[1]{\left\vert #1 \right\vert}
  \newcommand{\ps}[2]{\left\langle #1,#2 \right\rangle}
  \newcommand{\Matrice}{\mathcal{M}}
  \newcommand{\N}{\mathbb{N}}

  \newcommand{\C}{\mathbb{C}}
  \newcommand{\Spec}[1]{\sigma(#1)} 
  \newcommand{\D}{\mathbb{D}} 
  \newcommand{\T}{\mathbb{T}} 
  
  \newcommand{\NumRay}[1]{ w(#1) }
  \newcommand{\NumRayJoint}{w_J}
  \newcommand{\Reel}{\mathsf{Re}}
  
  \newcommand{\NumRho}[1]{ w_{#1} }
  \newcommand{\Aut}{\mathrm{Aut}}
  \newcommand{\Free}{\mathbb{F}}
  \newcommand{\conj}[1]{\overline{#1}}
  \newcommand{\adherence}[1]{\overline{#1}}

\numberwithin{equation}{section}
  
\title[Spectral sets and operator radii]{Spectral sets and operator radii}

 \author[C. Badea]{Catalin Badea}
   \address{(C. Badea) Univ. Lille, CNRS, Laboratoire Paul Painlev\'{e} - UMR 8524, Lille, France}
 \email{catalin.badea@univ-lille.fr}
\author[M. Crouzeix]{Michel Crouzeix}
 \address{(M. Crouzeix) Univ. Rennes, CNRS, IRMAR - UMR 6625, Rennes, France}
\email{michel.crouzeix@univ-rennes1.fr}
 \author[H. Klaja]{Hubert Klaja}
   \address{(H. Klaja) \'Ecole Centrale de Lille, CNRS, Laboratoire Paul Painlev\'{e} - UMR 8524, Lille, France}
 \email{hubert.klaja@gmail.com}

\thanks{This work was supported in part by
 the project FRONT of the French
National Research Agency (grant ANR-17-CE40-0021), by the Labex CEMPI (ANR-11-LABX-0007-01) and by a Project PEPS 2018 (CNRS)}

\begin{document}

\begin{abstract}
We study different operator radii of homomorphisms from an operator algebra into $\B{H}$ and show that these can be computed explicitly in terms of the usual norm. As an application, we show that if $\Omega$ is a $K$-spectral set for a Hilbert space operator, then it is a $M$-numerical radius set, where $M=\frac{1}{2}(K+K^{-1})$. This is a counterpart of a recent result of Davidson, Paulsen and Woerdeman. More general results for operator radii associated with the class of operators having $\rho$-dilations in the sense of Sz.-Nagy and Foia\c{s} are given. 
A version of a result of Drury concerning the joint numerical radius of non-commuting $n$-tuples of operators is also obtained.
\end{abstract}
\maketitle

\section{Introduction}
Let $H$ be a complex Hilbert space and let $\B{H}$ be the C$^*$-algebra of all bounded linear operators acting on $H$.  
The celebrated von Neumann inequality \cite{vN} states that 
$$ \norme{p(T)} \le \sup \lbrace \abs{p(z)} : z \in \D \rbrace $$ 
for every $T \in \B{H}$ with $\norme{T}\le 1$ and every polynomial $p$ with complex coefficients. Here $\|\cdot\|$ is the operator norm and $\D$ is the open unit disc. One says that $T$ is a (Hilbert space) \emph{contraction} and that $\D$ is a \emph{spectral set} for $T$. 

More generally, for $K\ge 1$, a set $\Omega$ is said to be a \emph{$K$-spectral set} for an operator $T$ if for every rational function $f$ with poles off of $\overline{\Omega}$, one has
\begin{equation}\label{eq:11}
\| f(T)\| \le K \|f\|_\Omega  .
\end{equation}
Here $\|f\|_\Omega = \sup\{ |f(z)| : z \in \Omega\}$. 
The set $\Omega$ is a \emph{complete $K$-spectral set} if \eqref{eq:11} holds for all matrices with rational coefficients.
The inequality \eqref{eq:11} extends to $R(\overline{\Omega})$, the uniform closure of these rational functions in $\mathrm{C}(\overline{\Omega})$.
When the complement of $\Omega$ is connected, it suffices to verify this inequality for polynomials.
We simply call $\Omega$ a \emph{(complete) spectral set} when the constant $K$ is equal to $1$.
It was proved by Paulsen \cite{Paulsen_1984} that $\Omega$ is a complete $K$-spectral set for $T$ if and only if there is an invertible operator $L$ so that $\Omega$ is a complete spectral set for $LTL^{-1}$ and $\|L\|\,\|L^{-1}\| \le K$. We cannot replace in this statement complete $K$-spectral sets by $K$-spectral sets as a counter-example of Pisier  \cite{Pisier_1997} (with $\Omega=\D$) shows.
We refer the reader to two surveys \cite{Paulsen1, BadBeck} and two books \cite{PaulsenBook, PisierBook} for more information about these notions.

Some counterparts of the von Neumann inequality, and variants of the notion of spectral sets, are possible with the operator norm replaced by the numerical radius. Recall that the 
numerical range of an operator $T\in \B{H}$ is the set
$$ W(T) = \left\{ \ps{Tx}{x} : \|x\|=1 \right\} $$
and the
\emph{numerical radius} of $T$, given by
$$ \NumRay{T} = \sup \left\{ |\lambda| : \lambda\in W(T) \right\},$$
is an equivalent Banach space norm ($w(T)\le \norme{T} \le 2w(T)$), but not a Banach algebra norm. It was recently proved by the second named author and Palencia \cite{CrPa} that for any Hilbert space operator $T$, 
the numerical range $W(T)$ is a complete $(1+\sqrt{2})$-spectral set (the optimal constant is conjectured to be $2$). 
Earlier results by Berger and Stampfli \cite{Berger_Stampfli_1967} (see also Kato \cite{Kato_1965}) show that
\begin{equation}\label{eq:BergerStampfli}
\NumRay{p(T)} \le \sup \lbrace \abs{p(z)} : z \in \D \rbrace 
\end{equation}
for every $T \in \B{H}$ such that $\NumRay{T} \le 1 $, 
and for every polynomial $p$ such that $p(0)=0$.
In the case when $p(0) \ne 0 $, Drury \cite{Drury_2008} proved that
\begin{equation}\label{eq:drury}
\NumRay{p(T)} \le \frac{5}{4} \sup \lbrace \abs{p(z)} : z \in \D \rbrace 
\end{equation}
and that $5/4$ is sharp. See also \cite{Klaja_Mashreghi_Ransford_2016} for an alternate proof. This motivates the following terminology: following \cite{Davidson_Paulsen_Woerdeman_2017} we say that a set $\Omega$ is a \emph{$\tilde{K}$-numerical radius set} for $T$ if 
$$ w(f(T)) \le \tilde{K} \|f\|_\Omega \quad \textrm{for all } f \in R(\Omega) ,$$
and it is a \emph{complete $\tilde{K}$-numerical radius set} if the same inequality holds for matrices over $R(K)$. 

Davidson, Paulsen and Woerdeman \cite{Davidson_Paulsen_Woerdeman_2017} recently proved that $\Omega $ is a complete $K$-spectral set for $T$ if and only if
$\Omega$ is a complete $\tilde{K}$-numerical radius set for $T$, where $\tilde{K}=\frac{1}{2}\left(K + K^{-1} \right)$. The proof given in \cite{Davidson_Paulsen_Woerdeman_2017} uses several facts which are specific for complete spectral sets and complete bounded maps. It is a natural question to ask if the analogue result for $K$-spectral sets and $\tilde{K}$-numerical radius sets is true. It is one of the aims of this note to prove that the answer to this question is positive. The methods used in this paper are different than those of \cite{Davidson_Paulsen_Woerdeman_2017}. We also consider other operator radii instead of the numerical radius. These operator radii are associated with the class C$_\rho$ of operators having $\rho$-dilations which was introduced by Sz.-Nagy and Foia\c{s} (see \cite{Nagy_Foias}). Also, similarly as in \cite{Davidson_Paulsen_Woerdeman_2017}, we consider operator radii for homomorphisms from any operator algebra into $\B{H}$, and show that these can be computed explicitly in terms of the usual norm. In fact, our methods work for homomorphisms from any unital Banach algebra satisfying the von Neumann inequality.

\subsection*{Organization} The paper is organized as follows. A proof that a set is $K$-spectral for $T$ if and only if it is a $\tilde{K}$-numerical radius set (Theorem \ref{ThMainSimplified}) is given in the next section. Although a more general result for the operator radii $w_{\rho}$ will be proved later on, we hope that the proof of the numerical radius case will highlight the simple ideas used in the proofs. In Section~3 we give the definition of operator radii of a unital homomorphism from an operator algebra (or a Banach algebra which satisfies the von Neumann inequality) into $\B{H}$. In the same section the operator radii $w_{\rho}$ are revisited and the main result (Theorem \ref{ThMainTh}) is stated. Section~4 is devoted to the proof of the main result while in Section~5 we give several applications of Theorem \ref{ThMainTh} to (complete) $K$-spectral sets, including an extension of Drury's result \eqref{eq:drury} for the $w_{\rho}$ radii. The last section contains an application of the main result to the joint numerical radius of non-commuting $n$-tuples of operators (a notion introduced in \cite{popescu}). A version of Drury's result is obtained in this non-commutative context.

\section{Spectral sets as numerical radius sets}
The following result, showing the equivalence of $K$-spectral sets and $\tilde{K}$-numerical radius sets, is a counterpart of the result from \cite{Davidson_Paulsen_Woerdeman_2017}. 

\begin{theorem}
\label{ThMainSimplified}
Let $T \in \B{H}$, let $\Omega$ be an open domain in the complex plane and let $K \ge 1$ be a real number. Denote $\tilde{K} = \frac{1}{2} \left( K  + K^{-1} \right)$.  The following assertions are equivalent:
\begin{enumerate}
\item $\Omega $ is a $K$-spectral set for $T$;
\item $\Omega $ is a $\tilde{K}$-numerical radius set for $T$.
\end{enumerate}
\end{theorem}

Note first that inverting the formula defining $\tilde{K}$ we obtain
$$ K = \tilde{K} + \sqrt{ \tilde{K}^2 - 1} .$$
Since a spectral set is also $(1+\varepsilon)$-spectral set and $\tilde{K} \to 1$ if and only if $K\to 1$, it is possible to assume, without loss of any generality, that $K > 1$. Denote 
$$\phi(z)=\frac{1+z}{1-z}$$ 
and 
$$\mathcal{D}_K:= \phi(K^{-1}\D)$$ 
the open disc with center $\frac{K^2+1}{K^2-1}$ and radius $\frac{2K}{K^2-1}$. An equivalent description of $\mathcal{D}_K$ is that the endpoints of the interval $\left(\frac{K-1}{K+1}, \frac{K+1}{K-1} \right)$ are diametrically opposed in $\mathcal{D}_K$.

The following lemma reformulates the definition of a $K$-spectral set. Recall first that the real part of an operator $A$ is defined as $\Reel (A) = (A+A^*)/2$. 

\begin{lemma}
The set $\Omega$ is $K$-spectral for $T$ if and only if $\Reel (f(T)) \ge 0 $ for all analytic functions $f:\Omega \rightarrow \mathcal{D}_K$.
\end{lemma}

\begin{proof}
 Note that the map $h\mapsto f=\phi\circ h$ realises a one to one correspondence between the analytic functions  $f:\Omega \rightarrow \mathcal{D}_K$ and 
$h:\Omega \rightarrow K^{-1}\D$. Clearly, $\Omega$ is a $K$-spectral set if and only if for all such functions $h$ one has $\|h(T)\|\leq 1$. Equivalently, this holds if and only if $I{-}h(T)^*h(T)\geq 0$. Hereafter $I$ is the identity operator. Note that 
\begin{align*}
\Reel f(T) &= \Reel \left((I{+}h(T))(I{-}h(T))^{-1} \right)\\
           &= (I{-}h(T)^*)^{-1}(I{-}h(T)^*h(T))(I{-}h(T))^{-1}.
\end{align*}
Therefore $\Reel (f(T)) \ge 0 $ if and only if $I{-}h(T)^*h(T)\geq 0$. 
\end{proof}
The analogue result for numerical radius spectral sets reads as follows.
\begin{lemma} 
The set $\Omega$ is a $K$-numerical radius set for $T$ if and only if $\Reel (f(T)) \ge -I $
for all analytic functions $f:\Omega \rightarrow \mathcal{D}_K$.
\end{lemma}

\begin{proof}
The functions $f$ and $h$ have the same meaning as in the previous proof. Notice that $\Omega$ is a $K$-numerical radius set for $T$ if and only if $w(h(T))\leq 1$ for all analytic functions $h$ bounded by $K^{-1}$. This is equivalent to say that $\Reel (I {-}h(T))\geq 0$ for all such $h$. Indeed, the statement 
$$h \textrm{ bounded by } K^{-1}$$
is equivalent to 
$$\zeta \,h \textrm{ bounded by } K^{-1} \textrm{ for every } \zeta \in \T .$$
Then we remark that 
\begin{align*}
\Reel (f(T){+}I)=2\,\Reel (I{-}h(T))^{-1}.
\end{align*}
Therefore $\Reel (f(T){+}I) \ge 0 $ if and only if $\Reel (I{-}h(T))\geq 0$. 
\end{proof}

Now we are ready to prove the announced result.

\begin{proof}[Proof of Theorem \ref{ThMainSimplified}]
  Let $K>1$. Denote $\tilde{K} = \frac{1}{2} (K + K^{-1}) $ and $\alpha = \frac{2\tilde{K}( K-1)}{(\tilde{K}-1)(K+1)}$.
  Then $\alpha > 0 $ and $\mathcal{D}_{\tilde{K}} = \alpha \mathcal{D}_K - 1$. It is easy to see that $f$ maps $\Omega$ into $\mathcal D_{K}$ if and only if $g:=\alpha\,f{-}1$ maps $\Omega$ into $\mathcal D_{\tilde{K}}$.
The condition $\Reel f(T)\geq 0$ is clearly equivalent to $\Reel (g(T))\geq -I$.
The theorem then follows from the two previous lemmata.
\end{proof}

\section{Operator radii of homomorphisms into $\B{H}$}
We start by introducing the definitions, notations and prior results that we shall need.
  
 \subsection*{Banach algebras satisfying von Neumann inequality}
A unital complex Banach algebra $\mathcal{A} $ is said to satisfy the \emph{von Neumann inequality} if for every $p \in \C[z]$ and every $a \in \mathcal{A}$ such that $\norme{a} \le 1 $, we have $\norme{p(a)} \le \sup \lbrace \abs{p(z)} : z \in \D \rbrace$. 
Here $p(a)$ is defined by the usual polynomial functional calculus in a unital Banach algebra. It follows from the von Neumann inequality for $\B{H}$ that every Banach algebra
which is an operator algebra, (i.e. which is isometrically isomorphic to a closed, not necessarily
self-adjoint, subalgebra of $\B{H}$), also satisfies the von Neumann inequality. The converse question (is it true that every unital Banach algebra which satisfies the von Neumann inequality is an operator algebra ?) is an open question. A non-unital Banach algebra satisfying the von Neumann inequality for polynomials in a
single variable, without constant term, and which is not (isomorphic to) an operator algebra can be found in  \cite{Dixon}.

\subsection*{Operator radii $w_\rho$} 
The class of operators ${\mathcal C}_\rho$ defined below has been introduced by
Sz.-Nagy  and Foia\c{s} (see \cite[Chapter 1]{Nagy_Foias} and the references therein). 
The operator $T\in \B{H}$ is said to be in the class ${\mathcal C}_\rho$ of operators 
admitting unitary $\rho$-dilations if there exists a larger space ${\mathcal H}\supset H$ and 
a unitary operator $U\in  \B{\mathcal H})$ such that 
\begin{equation*}
T^n=\rho P U^nP^*,\qquad \hbox{for } n=1,2,\dots .
\end{equation*}
Here $P$ denotes the orthogonal projection from ${\mathcal H}$ onto $H$. 
Then the operator $\rho$-radius of an operator $A$ is defined by
\begin{equation*}
w_\rho(A)=\inf\{\lambda>0\, : \lambda^{-1}A\in {\mathcal C}_\rho\}. 
\end{equation*}
From this definition it follows that $r(A)\leq w_\rho(A)\leq \rho \|A\|$, 
where $r(A)$ denotes the spectral radius of $A$. Also, $w_\rho(A)$ 
is a non-increasing function of $\rho$ and 
$$\lim_{\rho\to\infty} w_\rho(A) = r(A) .$$ 
Another equivalent definition follows from \cite[Theorem 11.1]{Nagy_Foias}:
\begin{align*}
w_\rho(A)&=\sup_{h\in \mathcal E_\rho}\Big\{(1\!-\!\tfrac1\rho)\,|\langle Ah,h\rangle|+\sqrt
{(1\!-\!\tfrac1\rho)^2|\langle Ah,h\rangle|^2+(\tfrac2\rho{-}1)\,\| Ah\|^2}\Big\},\quad\hbox{with}\\
 \mathcal E_\rho&=\{h\in H\,: \|h\|=1\ {\rm and } 
\  (1\!-\!\tfrac1\rho)^2|\langle Ah,h\rangle|^2
 -(1\!-\!\tfrac2\rho)\,\| Ah\|^2\geq0\}.
\end{align*}
Notice that $ \mathcal E_\rho=\{h\in H\, : \|h\|=1\}$ whenever $1\leq \rho\leq 2$. 
This shows that $w_1(A)=\|A\|$,
$w_2(A)=w(A)$ and $w_\rho(A)$ is a convex function of $A$ if $1\leq \rho\leq 2$. 

We shall need later on the following alternate characterisation of the condition $\NumRho{\rho}(T) \le 1$. 
\begin{proposition}
\label{ThKSpectrNumRho}
Let $T \in \B{H}$ with $\Spec{T} \subset \D$ and let $\rho \ge 1 $.  
We have $\NumRho{\rho}(T) \le 1$
if and only if
\begin{equation}\label{eq:poisson}
\Reel\big((I{-}\zeta T)^{-1}(I{+}\zeta T)\big) \ge (1- \rho)I
\end{equation}
for every $\zeta \in \T$.
\end{proposition}

\begin{proof}
The condition $\Spec{T} \subset \D$ assures that all operators below are well-defined. Using the relation 
$$ \Reel\big((I{-}\zeta T)^{-1}(I{+}\zeta T)\big) = \left( I-\zeta T\right)^{-1} +  \left( I-\conj{\zeta} T^*\right)^{-1} - I,$$
reminiscent of the relation between the Cauchy and Poisson kernels, the condition of Proposition \ref{ThKSpectrNumRho} says that 
$$ \left( I-\zeta T\right)^{-1} +  \left( I-\conj{\zeta} T^*\right)^{-1} - I \ge (1- \rho)I $$
for every $\zeta \in \T$. The proof follows now from \cite[p. 315]{CassierFack}.
\end{proof}

\subsection*{Operator radii of homomorphisms}
In this subsection $\omega$ stands for one of the operator radii $w_{\rho}$, $\rho\ge 1$.
If $\Phi:  \mathcal{A} \to \B{H}$ is a bounded linear map and $\omega$ is an operator radius we 
denote by 
$$ \|\Phi\|_{\omega} := \sup \{ \omega(\Phi(a)) : a \in  \mathcal{A}, \|a\|\le 1\}$$
the \emph{operator radius} of $\Phi$. In the case when $\omega = w_1$ is the classical norm we simply write $\|\Phi\|$.

Suppose now that $\mathcal{A}$ is a unital operator algebra. Then $\Matrice_n(\mathcal{A})$, the set of $n\times n$ matrices with entries in $\mathcal{A}$, has a canonical family of norms and $\mathcal{A}$ may be represented completely isometrically as an algebra of operators on some Hilbert space. See \cite{PaulsenBook} for details. If $\Phi:\mathcal{A} \to \B{H}$ is a bounded linear map, it induces coordinatewise maps 
$$\Phi^{(n)}: \Matrice_n(\mathcal{A}) \to \Matrice_n(\B{H}) \simeq \B{H^{(n)}} .$$ 
Here $ \Matrice_n(\B{H})$ is identified with the set of bounded linear operators acting on $H^{(n)}$, the $\ell_2$-sum of $n$ copies of $H$. One defines the \emph{completely bounded norm} by
$$\|\Phi\|_{cb} = \sup_{n\ge1} \|\Phi^{(n)}\| .$$
We also introduce the \emph{complete operator radius} as
$$ \|\Phi\|_{\omega cb} := \sup_{n\ge1}\,  \sup \, \{ \omega(\Phi^{(n)}(A)) : A \in \Matrice_n(\mathcal{A}),\ \|A\|\le 1\}.$$

These notions have been introduced in \cite{Davidson_Paulsen_Woerdeman_2017} in the case when $\omega = w_2$ is the numerical radius.

  \subsection*{Main result} The following result, which is the main result of this paper, generalizes Theorem \ref{ThMainSimplified} and the result from \cite{Davidson_Paulsen_Woerdeman_2017} to the larger class of operator radii $w_{\rho}$.
  
\begin{theorem}
\label{ThMainTh}
We assume $K\geq 1$ and $\rho \geq 1$ and set
\[
  \widetilde{K} =\frac{K^2+1+\sqrt{(K^2+1)^2-4\rho (2-\rho )K^2}}{2\rho K}.
\]
Then,\\
(a) \quad $\|\Phi\| = K$ if and only if $\|\Phi\|_{w_{\rho}} = \widetilde{K}$ for every unital homomorphism $\Phi: \mathcal{A} \to \B{H}$ from a Banach algebra $\mathcal{A}$ satisfying the von Neumann inequality ;\\
(b) \quad $\|\Phi\|_{cb} = K$ if and only if $\|\Phi\|_{w_{\rho}cb} = \widetilde{K}$ for every unital homomorphism $\Phi: \mathcal{A} \to \B{H}$ from an operator algebra $\mathcal{A}$.
\end{theorem}
 
\begin{remark}
Note that for $\rho =2$ the new notation becomes $ \widetilde{K} =\frac12(K{+}K^{-1})$ and coincides with the notation of Section\,2. Note also that $(K^2+1)^2-4\rho (2-\rho )K^2\geq 0$, thus $\widetilde{K} $ is real. It is easy to verify that $1\leq \widetilde{K} \leq K$, with strict inequalities if $K>1$ and $\rho >1$.
\end{remark}

\section{Proof of Theorem \ref{ThMainTh} }
As in the proof of Theorem \ref{ThMainSimplified} we may assume $K>1$. 
\begin{lemma}\label{LemDisqueDK}
Suppose $K>1$. The homothety defined by
\[
\phi_{\alpha,\rho }(z)=\alpha\,z+(1{-}\rho ), \quad\text{with }\alpha=\frac{\widetilde{K}}{\widetilde{K}^2-1}\,\frac{K^2-1}{K},
\]  
realises a biholomorphism from $\mathcal{D}_K $ onto $\mathcal{D}_{\widetilde{K}}$.
\end{lemma}

\begin{proof}
Recall that  $\mathcal{D}_K $ is the disc centered in $\omega_K=\frac{K^2+1}{K^2-1}$ with radius $r_K=\frac{2K}{K^2-1}$. We remark that $r_{\widetilde K}=\alpha\, r_K$. Thus it suffices to show that $\phi_{\alpha,\rho }(\omega_K )=\omega_{\widetilde K}$, i.e. 
\[
\frac{\widetilde{K}}{\widetilde{K}^2-1}\,\frac{K^2+1}{K}+1-\rho =\frac{\widetilde{K}^2+1}{\widetilde{K}^2-1},
\]
which is equivalent to
\[
\rho\, K\,\widetilde K^2-(K^2{+}1)\widetilde K +(2{-}\rho )K=0.
\]
Our choice of $\widetilde K$ clearly satisfies this equation.
\end{proof}

Let us introduce the function
$$
\varphi_K(z) = \varphi(\frac{z}{K}) = \frac{1 + \frac{z}{K}}{1 - \frac{z}{K}} .
$$
Note that $\varphi_K $ is a biholomorphism from $\D$ onto $\mathcal{D}_K $.

The closed unit ball of the Banach algebra $\mathcal{A}$ is denoted from now on as $\mathcal{B}$. It is easy to see that if $\mathcal{A} $ satisfies von Neumann inequality, then for every automorphism of the disc, $b$, we have $b(\mathcal{B}) = \mathcal{B}$. We denote by $\sigma_{\mathcal{A}}(f)$ the spectrum of the element $f\in \mathcal{A}$.  

\begin{proof}[Proof of Theorem \ref{ThMainTh}]

(a) \quad Recall that we suppose $K > 1$. Let $\phi_{\alpha,\rho} $ as in Lemma \ref{LemDisqueDK} and note that $\alpha > 0$. Using the inclusion
$$ \sigma_{\B{H}}(\Phi(f)) \subset \sigma_{\mathcal{A}}(f)$$
we obtain 
$$\sigma_{\B{H}}(\Phi(f)) \subset \adherence{\D} $$
for every $f\in \mathcal{B}$. Since $\varphi_K$ is a rational function with one pole outside $\adherence{\D}$, the operator 
$\varphi_K(\Phi(f))$ is well-defined. Similar arguments show that all operators below are also well-defined.

Using Proposition \ref{ThKSpectrNumRho} with $\rho = 1$, we see that $\|\Phi\| \le K$ is equivalent to the assertion that 
$$
\textrm{ for every } f\in\mathcal{B} 
 \textrm{ and every }  \zeta \in \T
\textrm{ we have }
\Reel(\varphi_K(\zeta \Phi(f))) \ge 0 .
$$ 
Note that $f \in \mathcal{B}$ if and only if $\zeta f \in \mathcal{B}$. 
Therefore the latter is equivalent to the assertion that
$$\textrm{ for every } f \in \mathcal{B} \textrm{ we have } 
\Reel(\varphi_K( \Phi(f))) \ge 0.
$$ 
This is equivalent to
$$
\Reel(\phi_{\alpha,\rho}(\varphi_K( \Phi(f)))) = \Reel(\alpha \varphi_K( \Phi(f)) + (1- \rho)I) \ge (1- \rho)I.
$$
Consider now the automorphism of the unit disc given by 
$$b=\varphi_{\widetilde{K}}^{-1} \circ \phi_{\alpha,\rho} \circ \varphi_K .$$
We obtain
$$
\Reel(\phi_{\alpha,\rho}(\varphi_K( \Phi(f)))) = \Reel(\varphi_{\widetilde{K}}(b (\Phi(f))) = \Reel(\varphi_{\widetilde{K}}(\Phi(b(f))) .
$$
As $b(\mathcal{B}) = \mathcal{B}$, this is equivalent to the assertion that
$$\textrm{ for every } g \in \mathcal{B},
\textrm{ we have }
\Reel(\varphi_{\widetilde{K}}(\Phi(g)) \ge (1- \rho)I.
$$
Note that $g  \in \mathcal{B}$ if and only if for every $\zeta \in \T $, $\zeta g \in \mathcal{B}$.
Therefore the latter is equivalent to the assertion that 
$$
\textrm{ for every } g \in \mathcal{B} \textrm{ and
for every } \zeta \in \T
\textrm{ we have }
\Reel(\varphi_{\widetilde{K}}(\zeta\Phi(g)) \ge (1- \rho)I.
$$
Using Proposition \ref{ThKSpectrNumRho}, we see that this is equivalent to the assertion that
$$
\textrm{ for every } g \in  \mathcal{B}, \textrm{ we have }
\NumRho{\rho}(\Phi(g)) \le \widetilde{K}.
$$
This finishes the proof of (a). 

The proof of (b), for complete bounded norms of homomorphisms of operator algebras, follows from (a) and the fact that all matrix algebras $\Matrice_n(\mathcal{A})$ satisfies in this case the von Neumann inequality. 
\end{proof}

\section{Applications to spectral sets}

\subsection*{$K$-spectral sets}
Let $\Omega \subset \C$ be an open domain and let 
$T \in \B{H}$ with $\Spec{T} \subset \overline{\Omega}$.
Let $\mathcal{A} = R(\overline{\Omega})$ be the uniform closure in C$(\overline{\Omega})$ of rational functions with poles off of  $\overline{\Omega}$. As a subalgebra of a (commutative) C$^*$-algebra, $\mathcal{A} $ is an operator algebra. 
Consider now the map $\Phi: \mathcal{A} \mapsto \B{H}$ given by $\Phi(f) = f(T)$. 
The notion of operator radii for the unital homomorphism $\Phi$ of the rational functional calculus defined above leads to a generalisation of $K$-spectral sets and $K$-numerical radius sets. We say that $\Omega $ is a $K$-$\NumRho{\rho}$ \emph{set} for $T$ if 
for every rational function $f$ we have 
$$
  \NumRho{\rho}(f(T)) \le K \|f\|_{\Omega}.
$$
The analogue notion of a \emph{complete} $K$-$\NumRho{\rho}$ \emph{set} is defined in a similar way.
The case $\rho=1$ corresponds to the usual notion of $K$-spectral sets while $K$-$\NumRho{2}$ sets are the $K$-numerical radius sets from  \cite{Davidson_Paulsen_Woerdeman_2017}. 
We deduce the following result from Theorem \ref{ThMainTh}. 
\begin{theorem}
\label{ThSpectalSetKtoKtilde}
Let $\Omega$ be a domain of the complex plane and let $K \ge1$ and $\rho \ge 1$. Set 
\[
  \widetilde{K} =\frac{K^2+1+\sqrt{(K^2+1)^2-4\rho (2-\rho )K^2}}{2\rho K}.
\]
  Then $\Omega$ is a $K$-spectral set for $T$ if and only if $\Omega$ is a $\widetilde{K}$-$\NumRho{\rho} $ spectral set for $T$. Also, $\Omega$ is a complete $K$-spectral set for $T$ if and only if $\Omega$ is a complete $\widetilde{K}$-$\NumRho{\rho} $ spectral set for $T$.
\end{theorem}

\subsection*{An extension of a result of Drury} 
As mentioned in Introduction, Drury \cite{Drury_2008} proved that if $\NumRay{T} \le 1 $, then $\D $ is a $\frac{5}{4}$-numerical radius set for $T$. We generalise this result for the operator radius $w_{\rho}$ as follows.

\begin{theorem}
  Let $\rho \ge 1$ and  let $T \in \B{H} $ be such that $\NumRho{\rho}(T) \le 1$.
  Denote
  $$
  C = \frac{\rho^2 + 1 + (\rho - 1)\sqrt{5\rho^2+2\rho+1}}{2\rho^2}
  $$
  Then $\D $ is a complete $C$-$\NumRho{\rho}$ spectral set for $T$.  
\end{theorem}

\begin{proof}
It was proved by Okubo and Ando \cite{Okubo_Ando_1975} that $\D $ is a complete $\rho$-spectral set for $T$ if $\NumRho{\rho}(T) \le 1$. We apply now Theorem \ref{ThSpectalSetKtoKtilde} with $K=\rho$. 
\end{proof}

\section{Joint numerical radius of non-commuting $n$-tuples of operators}

As a different application of the main result we consider in this section non-commuting $n$-tuples of operators with joint numerical radius at most one, a notion introduced in \cite{popescu}. In the same paper Popescu proved, among other things, several analogues of univariate results, including an analogue of the Berger's dilation theorem and of the Berger-Stampfli-Kato theorem \eqref{eq:BergerStampfli} (see \cite[Cor. 1.11]{popescu}). We show here how Theorem \ref{ThMainTh} implies in this non-commutative setting analogues of the Okubo-Ando result mentioned above (for $\rho=2$) and of Drury's result \eqref{eq:drury}. This partially improves \cite[Cor. 1.9]{popescu}.

We need the following definitions. We refer to \cite{popescu,popescu1} and the references therein for more information. Let $n \in \N^*$.
Denote by $\Free^+_n$ the free semigroup
with $n$ generators $g_1, g_2, \dots , g_n$ and the identity element $g_0$.
For $\alpha \in \Free^+_n$ we define 
$$
\abs{\alpha} = \left\lbrace\begin{array}{ll}
k & \textrm{if } \alpha = g_{i_1}g_{i_2} \dots g_{i_k} \\ 
0 & \textrm{if } \alpha = g_0 .
\end{array}   \right.
$$
Let $T_1, T_2, \dots , T_n \in \B{H}$.
The \emph{joint numerical radius} of $(T_1, T_2, \dots , T_n)$ is defined by 
$$
\NumRayJoint(T_1, \cdots , T_n) = \sup \abs{\sum_{\alpha \in \Free^+_n} \sum_{j=1}^n \ps{h_\alpha}{T_j h_{g_j\alpha}}},
$$
where the supremum is taken over all families of of vectors $\lbrace h_ \alpha, \alpha \in \Free^+_n \rbrace \subset H $ 
with 
$$\sum_{\alpha \in \Free^+_n} \norme{h_\alpha}^2 = 1 .$$
It has been proved in \cite{popescu} that this notion coincides with the classical one for $n=1$. This also follows from the formula \cite[Cor. 1.2]{popescu}
$$ \NumRayJoint(T_1, \cdots , T_n) = w(S_1\otimes T_1^* + \cdots + S_n\otimes T_n^*),$$
where $S_1, \cdots, S_n$ are the left creation operators defined below. Notice however that, in general, $\NumRayJoint(T_1, T_2, \dots , T_n) \neq \NumRayJoint(T^*_1, T^*_2, \dots , T^*_n)$ if $n\ge 2$.

Let $H_n $ be a $n$ dimensional Hilbert space with orthonormal basis $e_1, e_2, \dots , e_n $.
The full Fock space of $H_n$ is 
$$
  F^2(H_n) = \bigoplus_{k \ge 0} H_n^{\otimes k},
$$
where  $H_n^{\otimes 0} = \C 1$, and $ H_n^{\otimes k} $ is the Hilbert tensor product of $k$ copies of $H_n$. Denote the norm of $F^2(H_n)$ as $\|\cdot\|_2$. 

The left creation operators $S_i : F^2(H_n) \rightarrow F^2(H_n) $ 
are defined for every $\varphi \in F^2(H_n)$ by 
$$S_i \varphi = e_i \otimes \varphi .$$

Set
$$
e_{\alpha} = \left\lbrace\begin{array}{ll}
e_{i_1} \otimes e_{i_2} \otimes \dots \otimes e_{i_n} & \textrm{if } \alpha = g_{i_1}g_{i_2} \dots g_{i_k} \\ 
1 & \textrm{if } \alpha = g_0 .
\end{array}   \right.
$$
Denote by $\mathcal{P}_n$ the set of all $p$ in $ F^2(H_n) $ of the form
$$
p = \sum_{\abs{\alpha} \le m} a_\alpha e_\alpha
$$
with $m \in \N$ and $a_\alpha \in \C $. The set $\mathcal{P}_n$ may be viewed as the algebra of polynomials in $n$ non-commuting indeterminates, with $p\otimes q$, $\, p,q \in  \mathcal{P}_n$, as multiplication. Define now $F_n^{\infty}$ as the set of all $g\in F^2(H_n)$ such that 
$$ \|g\|_{\infty} := \sup \{\|g\otimes p\|_2 : p\in \mathcal{P}_n, \|p\|_2 \le 1\} < \infty .$$
We denote by $\mathcal{A}_n$ the closure of $\mathcal{P}_n$ in $(F_n^{\infty},\|\cdot\|_{\infty})$. The Banach algebra $\mathcal{A}_n$, which can be viewed as a non-commutative analogue of the disc algebra, is an operator algebra (see for instance \cite{popescu1}). 
Denote 
$$
T_{\alpha} = \left\lbrace\begin{array}{ll}
T_{i_1}T_{i_2}\dots T_{i_n} & \textrm{if } \alpha = g_{i_1}g_{i_2} \dots g_{i_k} \\ 
I & \textrm{if } \alpha = g_0 .
\end{array}   \right.
$$
For $p \in \mathcal{P}_n$ denote by $p(T_1,T_2,\dots ,T_n)$ the operator $\sum_{\abs{\alpha} \le m} a_\alpha T_\alpha $. Notice that $\norme{p}_\infty \le 1$ if and only if $\norme{p(S_1, \dots, S_n)} \le 1$ (see for instance \cite[Thm. 3.1]{popescu1}). We consider the map
$$ \Psi : \mathcal{A}_n \mapsto \B{H}, \quad \Psi(p) = p(T_1, \cdots, T_n) .$$

We can now state the analogues of the results of Okubo-Ando and Drury. 

\begin{theorem}\label{thm:pop}
Let $T_1, \cdots , T_n \in \B{H}$ with joint numerical radius $w_{J}(T_1, T_2, \dots , T_n) \le 1$. Then $\|\Psi\| \le 2$ and $\|\Psi\|_{w} \le \frac{5}{4}$.
\end{theorem}
We first need the following lemma.
\begin{lemma}
  Let $p \in \mathcal{P}_n$.
  If $p$ is non constant and $\norme{p}_{\mathcal{A}_n} \le 1 $, then $\abs{p(0,\dots,0)} < 1 $.
\end{lemma}
\begin{proof}
  Let 
$$
p = \sum_{\abs{\alpha} \le m} a_\alpha e_\alpha
$$  
  be non constant.
  We have 
  \begin{align*}
  \norme{p}_{\mathcal{A}_n}^2
   &\ge \norme{p(S_1,S_2, \dots, S_n) e_0}_{F^2(H_n)}^2 \\
   &= \norme{\sum_{\abs{\alpha} \le m} a_\alpha S_\alpha e_0}_{F^2(H_n)}^2 \\
   &= \norme{\sum_{\abs{\alpha} \le m} a_\alpha e_\alpha}_{F^2(H_n)}^2 \\
   &= \sum_{\abs{\alpha} \le m} \abs{a_\alpha} ^2 .
 \end{align*}
 Note that $p(0, \dots, 0) = a_0 $. 
 As $p$ is non constant, there exists $\alpha \ne g_0$ such that $a_\alpha \ne 0$.
 Suppose, for the sake of contradiction, that $\abs{p(0, \dots, 0)} \ge 1$. Then 
 $$\norme{p}_{\mathcal{A}_n} \ge \sqrt{\abs{p(0)}^2 + \abs{a_\alpha}^2} > 1 .$$
 This contradicts the hypothesis $\norme{p}_{\mathcal{A}_n} \le 1 $, so  $\abs{p(0,\dots,0)} < 1 $.
\end{proof}
\begin{proof}[Proof of Theorem \ref{thm:pop}]
Let
  $$
  p = \sum_{\abs{\alpha} \le m} a_\alpha e_\alpha 
  $$
  be a non constant polynomial in the non-commutative disc algebra $\mathcal{A}_n$ 
  such that 
  $$\norme{p(S_1, S_2, \dots , S_n)} \le 1.$$
  Denote 
  $$
  q = \sum_{\abs{\alpha} \le m, \alpha \ne 0} a_\alpha e_\alpha.
  $$
  Then $p = a_0 e_0 + q $.
  Note that $a_0 \in \D $.
  Denote 
  $$
  b(z) = \frac{z - a_0}{1-\conj{a_0}z}. 
  $$
  Then $b \in \Aut(\D)$ and $ p = b^{-1} (h)$, where $h = b(p)$.
  Then
  \begin{align*}
    h &= b(p) \\
      &= (p- a_0) (1 - \conj{a_0}p)^{-1} \\
      &= q  \sum_{k=0}^\infty (\conj{a_0}p)^k \\ 
      &= \lim_{N \rightarrow \infty} q \sum_{k=0}^N (\conj{a_0}p)^k .
  \end{align*}   
 Set $h_N =  q \sum_{k=0}^N (\conj{a_0}p)^k $.
  Then $h_N(0, \dots, 0) = 0 $ and so, by the analogue of the Berger-Stampfli-Kato mapping theorem proved in \cite[Cor. 1.11]{popescu}, we have 
  $$\NumRay{h_N(T_1, T_2, \dots, T_n)} \le \norme{h_N(S_1, S_2, \dots, S_n)}.$$
  Letting $N \rightarrow \infty$ we get 
  $$
   \NumRay{h(T_1, T_2, \dots, T_n)} \le \norme{b(p(S_1, S_2, \dots, S_n))} \le 1 .
  $$
Using Drury's result from \cite{Drury_2008}, we obtain
  $$
  \NumRay{p(T_1, T_2, \dots, T_n)} = \NumRay{b^{-1}(h(T_1, T_2, \dots, T_n))} \le \frac{5}{4} .
  $$
Therefore
  $$
  \| \Psi \|_{w} \le \frac{5}{4}.
  $$
Applying now Theorem \ref{ThMainTh} we get
  $$ 
  \|\Psi\| \le 2 .
  $$

The proof is complete.
\end{proof}

\bibliographystyle{plain}
\bibliography{BaCrKl-final}
\end{document}